\documentclass[psamsfonts]{amsart}

\usepackage{amssymb,amsfonts}
\usepackage[all,arc]{xy}
\usepackage{enumerate}
\usepackage{mathrsfs}
\usepackage{esint}

\newtheorem{thm}{Theorem}[section]
\newtheorem{cor}[thm]{Corollary}
\newtheorem{prop}[thm]{Proposition}
\newtheorem{lem}[thm]{Lemma}

\theoremstyle{definition}
\newtheorem{defn}[thm]{Definition}

\theoremstyle{remark}
\newtheorem{rem}[thm]{Remark}

\makeatletter
\let\c@equation\c@thm
\makeatother
\numberwithin{equation}{section}

\title[Parabolic system with degenerate coefficients]{Global solution for a coupled parabolic system with degenerate coefficients and time-weighted sources}

\author{Ricardo Castillo}
\address{Departamento de Matem\'atica, Facultad de Ciencias, Universidad del B\'io-B\'io, Concepci\'on, Chile}
\curraddr{}
\email{rcastillo@ubiobio.com}
\thanks{R. Castillo was supported by ANID-FONDECYT project No. 11220152 }

\author{Omar Guzm\'an-Rea}
\address{Departamento de Matem\'atica, Universidade de Bras\'ilia, Brasil\'ia - DF, Brazil}
\curraddr{}
\email{omar.grea@gmail.com}
\thanks{O. Guzm\'an-Rea was supported by CNPq/Brazil, 166685/2020-8}

\author{Miguel Loayza}
\address{Departamento de Matem\'atica, Universidade Federal de Pernambuco, Recife, Pernambuco, Brazil}
\curraddr{}
\email{miguel.loayza@ufpe.br}
\thanks{M. Loayza was partially supported by CAPES-PRINT, 88881.311964/2018-01, MATHMSUD, 88881.520205/2020-01, 21-MATH-03}

\author{Mar\'ia Zegarra}
\address{Departamento de Matem\'atica, Universidad Nacional Mayor de San Marcos, Lima, Per\'u}
\curraddr{}
\email{mzegarra@unmsm.edu.pe}
\thanks{}


\begin{document}

\begin{abstract}

In this paper, we obtain the so-called Fujita exponent to the following parabolic system with time-weighted sources and degenerate coefficients $  u_{t}- \mbox{div} ( \omega(x)\nabla { u} )= t^{r} v^{p} $ and $ v_{t}-  \mbox{div} ( \omega(x)\nabla {v} )= t^{s} u^{p}$  in  $\mathbb{R}^{N} \times (0,T)$ with initial data belonging to $ \left[L^\infty(\mathbb{R}^N)\right]^2.$ Where $p,q > 0$  with $ pq > 1$; $r,s>-1 $; and either $\omega(x) = | x_1|^{a},$ or $\omega(x) = | x |^{b}$ with $a,b > 0$. 

\end{abstract}

\maketitle

\tableofcontents

\section{Introduction}\label{sec1}
Many problems that emerge in several branches of science are associated with elliptic and parabolic partial differential equations, which present a diffusion operator of the form $\mbox{div}(\omega(x) \nabla  \cdot)$. Where $\mbox{div}$ is the divergent, $\nabla$ is the gradient, and the spatial function $\omega: \mathbb{R}^N \rightarrow [0,\infty)$ is a weight representing the part of thermal diffusion, which can degenerate (see, e.g., \cite{Bass, Crank, Epstein, Epstein1, Hamilton, Heston, Kamin3, Kamin4, Kamin5, Murray1, Murray2, WZ1, WZ2, Zeldovich} and the references therein). Several authors have extensively    studied models related to those problems, and the literature is well known. For example, see the works of Kohn and Nirenberg \cite{KN}; Fabes, Kenig, and Serapioni \cite{FKS}; Gutierrez and Nelson \cite{MR919445}; Fujishima, Kawakami, and Sire \cite{Fujish}; Dong and Phan \cite{Dong}; Sire, Terracini, and Vita \cite{STV}.

We are interested in the following degenerate parabolic problem with time-weighted sources.
\begin{equation}\label{eq2.1}
 \left\{\begin{array}{rlll}
 u_{t}-  \mbox{div} ( \omega(x)\nabla u )  & = & t^r v^{p}& \mbox{ in } \mathbb{R}^{N} \times (0,T),\\
  v_{t}- \mbox{div} ( \omega(x)\nabla v )   &= & t^s u^{q}& \mbox{ in } \mathbb{R}^{N} \times (0,T),\\
u(0)=  u_{0}   &    & v(0)=  v_{0} &  \mbox{ in } \mathbb{R}^{N},
      \end{array}\right.
\end{equation}
where  $(u_0,v_0) \in L^{\infty}(\mathbb{R}^N)\times L^{\infty}(\mathbb{R}^N) \equiv [L^\infty(\mathbb{R}^N)]^2$; $u_0, v_0 \geq 0$; $p,q > 0$ with $ pq > 1$; $r,s > -1$; and the weighted function $\omega : \mathbb{R}^N \rightarrow [0,\infty)$  either
\begin{itemize}
\item[$(A)$] $\omega(x) = | x_1 |^{a}$ with $a \in [0, 1)$ if $N = 1,2$; and $a \in [0, 2/N)$ if $N \geq 3$, or

\item[$(B)$]  $\omega(x) = | x |^{b}$ with $b \in [0, 1)$.
\end{itemize}
 Note that the function $\omega$ with these characteristics belongs to the Muckenhoupt class of functions \footnote{The Muckenhoupt classes $A_p$, with $p>1$, is defined as the class of locally integrable nonnegative functions $w$ that satisfies $\left( \displaystyle\fint_Q w dx \right) \left( \displaystyle\fint_Q w^{- \frac{1}{(p-1)}} \right)< K$ for every cube $Q$ and some constant $K$.} $A_{1 + \frac{2}{N}}$, and the operator $div(\omega(x)\nabla  \cdot)$ is not self-adjoint (Fujishima et al. in \cite[p.6]{Fujish} comment on this particularity).

When $\omega(x)=| x_1 |^{a},$ it admits a line of singularities; thus, the problem (\ref{eq2.1}) is related to the fractional Laplacian through the Caffarelli-Silvestre extension, see \cite{Caffarelli}, \cite{Sande}, and \cite{Fujish}. The fractional Laplacian is associated with nonlocal diffusion and appears in the Levy diffusion process; for example, see \cite{Daoud, Klafter}.

In \cite{Fujish}, Fujishima et al. studied the following problem
\begin{equation}\label{In.Fu}
\left\{\begin{array}{rlll}
 W_{t}-  \mbox{div} ( \omega(x)\nabla W )  & = &  W^{p}& \mbox{ in } \mathbb{R}^{N} \times (0,T),\\
W(0)&=&  W_{0}    & \mbox{ in } \mathbb{R}^{N},
      \end{array}\right.
\end{equation}
and obtained the following Fujita exponent
$$  p^{\star}(\alpha) =1 + \frac{2 - \alpha}{N}, $$
where $\alpha = a$ in case $(A)$ and $\alpha = b$ in case $(B)$.

Several authors extensively studied problem \eqref{In.Fu} when $\omega(x)=1$. Hirose Fujita pioneered the approach of associating a critical exponent with the global existence of solutions  \cite{FU}. Specifically, he showed that if $1<p<p^*(0)$, then problem \eqref{In.Fu} does not admit any non-negative global solution. For $p>p^*(0)$, there exist both global and nonglobal solutions, depending on the size of the initial data; see also \cite{ Levine, Souplet} for more details. In the critical case $p=p^*(0)$, Hayakawa \cite{Hayakawa} (when $N=1,2 $), and later Aronson and Weinberger \cite{Aronson} (when $N \geq 3$) showed that problem \eqref{In.Fu} has no global solution.

When $\omega(x)=1$ and $r=s=0$, problem \eqref{eq2.1} was studied by Escobedo and Herrero \cite{MR1088342}; they showed that
$$(pq)^* = 1 + \frac{2}{N} (\max\{p, q \} + 1)$$
is the Fujita exponent for the problem \eqref{eq2.1}. That means that if $pq>(pq)^*$ then any non-trivial nonnegative solution blows up in finite time, and when $pq\leq (pq)^*,$ there exist both global and nonglobal solutions.

Later, when $\omega(x)=1$ and the time-weighted sources are $(1+t)^r$ and $(1+t)^s$ instead of $t^r$ and $t^s$, respectively, Cao et al. \cite{Cao} showed the existence of the following Fujita exponent
$$(pq)^* = 1 + \frac{2 \max\{ (r+1)q + s +1, (s+1)p + r + 1 \}}{N},$$
for the problem \eqref{eq2.1}. See also \cite{Castillo}, \cite{RM}, \cite{Ishige} and the references therein for other related results.

The main contribution of the current work is to guarantee the existence of the so-called Fujita exponent for the problem \eqref{eq2.1}; for this purpose, we adapted the approach used in \cite{Fujish} and \cite{MR1088342}. Nevertheless, difficulties inherent to the degenerate coupled system \eqref{eq2.1} appear, and the case $pq>1$ with $0<p<1$ (or $0<q<1$) merit more effort. Also, note that when $p=q>1$, $r=s=0,$ and $u=v $, \eqref{eq2.1} is reduced to problem \eqref{In.Fu} studied recently in \cite{Fujish}.

We want to mention that our approach can be applied to study the critical Fujita exponent of the following coupled systems:
\begin{equation}\label{Coupled1}
\left\{\begin{array}{rllll}
  \frac{d u_i}{dt}  -  \mbox{div} ( \omega(x)\nabla u_i  )  & = &  t^{r_i} u^{q_i}_{i+1}& i=1,...,m-1 & \mbox{ in } \mathbb{R}^{N} \times (0,T),\\
\frac{d u_m}{dt}  -  \mbox{div} ( \omega(x)\nabla u_m  )&=&  t^{r_m} u^{q_m}_{1}&    & \mbox{ in } \mathbb{R}^{N} \times (0,T),
      \end{array}\right.
\end{equation}
~~\\
and
~~\\
\begin{equation} \label{Coupled2}
\left\{\begin{array}{rllll}
 {u}_t  -  \mbox{div} ( \omega(x)\nabla u  )  & = &  t^{r_1} u^{p} + t^{r_2} v^{q}  & \mbox{ in } \mathbb{R}^{N} \times (0,T),\\
{v}_t  -  \mbox{div} ( \omega(x)\nabla v  )&=&  t^{r_3} u^{r}  + t^{r_4} v^{s}   & \mbox{ in } \mathbb{R}^{N} \times (0,T)  .
      \end{array}\right.
\end{equation}
When $\omega(x)=1,$ problem \eqref{Coupled1} was studied in \cite{Renclawowicz}, \cite{Umeda} and \cite{RM}, whereas that \eqref{Coupled2} was studied in \cite{Cui}, \cite{Snoussi}, and \cite{Castillo}.
~~\\

Solutions to the problem \eqref{eq2.1} are understood in the following sense.
\begin{defn}\label{DEF} Let $u$ and $v$, a.e. finite, measurable functions in $\mathbb{R}^N \times (0.T) $ for some $T>0$. Then we call that $(u,v)$ is a solution of (\ref{eq2.1}), if $(u,v) \in L^\infty ((0,T); L^{\infty}(\mathbb{R}^N)) \times L^\infty ((0,T); L^{\infty}(\mathbb{R}^N))$ and satisfies
\begin{equation} \label{Mild}
\begin{array}{rllllll}
 u(t) &=& S(t) u_{0} + \displaystyle \int_{0}^{t} S(t - \sigma) \sigma^{r} v^{p}(\sigma) d \sigma,
 \\ \noalign{\medskip}
 v(t) &=& S(t) v_{0} + \displaystyle \int_{0}^{t} S(t - \sigma) \sigma^{s} u^{q}(\sigma)d\sigma,
\end{array}
\end{equation}
for almost $x \in \mathbb{R}^N$ and $t>0$. If $T= \infty $, then we say that $(u,v)$ is a global-in-time solution. Where $ S(t)\phi(x) := [S(t)\phi](x) := \displaystyle \int_{\mathbb{R}^{N}} \Gamma(x, y, t) \phi(y) dy $, here $\Gamma(x, y, t)$ is the fundamental solution of \eqref{FP}.
\end{defn}

In what follows, we consider the following values
\begin{align}
 \gamma_1 &:= \frac{(r+1) + (s+1) p}{pq-1} \label{V1}\\
 \gamma_2 &: = \frac{(s+1) + (r+1) q}{pq-1} \label{V2}\\
 r_{1 \star} &: = \frac{N}{(2 - \alpha) \gamma_1} \label{V3} \\
 r_{2 \star} &: = \frac{N}{(2 - \alpha) \gamma_2}. \label{V4}
\end{align}

Our main result is the following.
\begin{thm} \label{Theorem3} Let $r,s>-1,$ \, $p,q >0,$ \, $p \cdot q > 1$, and the values \eqref{V1}-\eqref{V4}.
Suppose $ \alpha = a $ in the case that $\omega$ satisfies the condition $(\mbox{A})$, and $ \alpha=b$ in the case that $\omega$ satisfies the condition $(\mbox{B})$.
\begin{enumerate} %[(i)]
  \item[(i)] If $ \gamma : = \max \left\{ \gamma_1, \gamma_2 \right\}  \geq \frac{N}{2 - \alpha},$ then problem $(\ref{eq2.1})$ has no nontrivial global-in-time solutions.
  \item[(ii)] If $ \gamma : = \max \{ \gamma_1, \gamma_2 \}   < \frac{N}{2 - \alpha},$ then there exists nontrivial global-in-time solutions to $(\ref{eq2.1})$. Also, there exists a constant $\delta > 0$ such that for any
      $$(u_{0}, v_{0}) \in [L^{\infty}(\mathbb{R}^N) \cap L^{r_{1 \star},\infty}(\mathbb{R}^N) ] \times [L^{\infty}(\mathbb{R}^N) \cap L^{r_{2 \star},\infty}(\mathbb{R}^N) ] $$
      with
      $$ \max \{ \|u_{0}\|_{r_{1 \star}, \infty},  \|v_{0}\|_{r_{2 \star}, \infty} \} < \delta,$$
      problem $(\ref{eq2.1})$ has a global-in-time solution $(u,v) $ satisfying:
\begin{equation*}
\sup_{t>0} (1+t)^{\frac{N}{2-\alpha} \left( \frac{1}{r_{1 \star}} - \frac{1}{\mu} \right)} \|u(t)\|_{\mu, \infty} < \infty
\end{equation*}
and
\begin{equation*}
 \sup_{t>0} (1+t)^{\frac{N}{2-\alpha} \left( \frac{1}{r_{2 \star}} - \frac{1}{\mu} \right)} \|v(t)\|_{\mu, \infty} < \infty
  \end{equation*}
for all $0< \mu $ such that $\max \{ r_{1 \star} , r_{2 \star} \} < \mu \leq \infty $.
\end{enumerate}
\end{thm}

\begin{rem} Here are some comments on Theorem \ref{Theorem3}.
\begin{itemize}
\item[(i)] When $\alpha=0$, Theorem \ref{Theorem3} coincides with result in \cite[Theorem 1]{Cao}.
\item[(ii)] When $\alpha = 0$ and $r=s=0$, this Theorem agrees with the results that appear in \cite{MR1088342}.
\item[(iii)]  This result shows the existence of the critical value of Fujita and is given by
$$(pq)^*(\alpha) = 1 + \frac{(2-\alpha) \max\{(s+1)p + r+ 1, (r+1)q + s + 1 \}}{N}$$.
\end{itemize}
\end{rem}

The work is organized in the following way. In section 2, we present the necessary preliminaries. In section 3, we prove the non-global existence. Finally, in section 4, we demonstrate global existence.

\section{Preliminaries and toolbox}

In that follows, $C$ denotes a generic positive constant that may vary in different places, and its change is not essential to the analysis.

The positive part of $\phi(x) $ is defined by
$$\phi^{+}(x)= \max \{\phi(x), 0 \}. $$
The negative part of $\phi $ is defined analogously.

Here $x_1$ is the first coordinate of $x = (x_1,...,x_N) \in \mathbb{R}^N$, and $ \mid . \mid$ is the Euclidean norm of $\mathbb{R}^N$. The spaces $L^\infty(\mathbb{R}^N)$ and $L^\zeta(\mathbb{R}^N) (\zeta\geq 1)$ are defined as usual, and its norms are denoted by $\| \cdot \|_\infty$ and $\| \cdot \|_\zeta$, respectively. %Let $X$ is a Banach space and $0<T\leq \infty$, we denote by $L^\infty(0,T;X)$ as Banach space of all measurable function $\varrho : (0,T) \rightarrow X$ such that $t \mapsto \|\varrho(t)\|_X$ belong to $L^\infty(0,T)$ with the norm
%$$
%\|\varrho\|_{L^\infty(0,T;X)} = \big\| \|\varrho(t) \big\|_X \|_{L^\infty(0,T)}.
%$$

The Lorentz space $L^{\zeta,\infty}(\mathbb{R}^N), \zeta >1$, is defined as follows
\begin{equation*}
L^{\zeta, \infty} := \left \{ \psi : \mathbb{R}^N \rightarrow \mathbb{R}, \| \psi \|_{L^{\zeta , \infty}(\mathbb{R}^N)} = \sup_{\rho >0} \mid  \mu \left\{ x \in \mathbb{R}^N : \mid \psi(x) \mid > \rho \right \} \mid^{1/\zeta} < \infty \right \},
\end{equation*}
where $\mu$ is Lebesgue measure $\mathbb{R}^N$ (see \cite{grafakos}).

We will denote by $\Gamma := \Gamma(x,y,t)$ the fundamental solution of the following homogeneous problem
\begin{equation}\label{FP}
W_{t}- \mbox{div} ( \omega(x) \nabla  W )   = 0
\end{equation}
in $\mathbb{R}^{N}\times (0,T) $, with a pole at point $(y, 0)$, with $\omega$ fulfilling either $(A)$ or $(B)$. Since $\omega(x)$ belongs to the class $A_{1+ \frac{2}{N}}$ of Muckenhoupt functions (see, e.g. \cite{Mum}), we have that the fundamental solution $\Gamma = \Gamma (x, y, t)$ verifies the following properties (for more details, see \cite{MR919445} and \cite{Fujish}):

\begin{itemize}
\item[$(\mbox{K}_{1})$] $ \displaystyle \int_{\mathbb{R}^{N}} \Gamma( x , y , t) dx=\int_{\mathbb{R}^{N}} \Gamma(x , y , t) dy=1$ for $x, y \in \mathbb{R}^{N} $ and $t > 0$;
\item[$(\mbox{K}_{2})$] $\displaystyle \Gamma(x , y , t) = \int_{\mathbb{R}^{N}} \Gamma(x, \xi , t - s) \Gamma(\xi , y , s) d\xi  $ for $x, y \in \mathbb{R}^{N} $ and $t > s > 0$;
\item[$(\mbox{K}_{3})$] Suppose that $c_0 := \sup_{Q} \left( \frac{1}{\mid Q \mid} \int_{Q} \omega(x) dx \right) \left( \frac{1}{\mid Q \mid} \int_{Q} \omega(x)^{-1} dx \right) < \infty,$ where the supremum is taken over all cubes $Q \in \mathbb{R}^{N},$ and
    $$ h_{x}(r) = \left( \int_{B_{r}(x)} \omega(y)^{-\frac{N}{2}} dy \right)^{\frac{2}{N}}.$$
    Then there exist constants $ C_{0\star},c_{0\star}>0$, depending only on $N$ and $c_0$, such that
\begin{eqnarray*}
&c_{0\star}^{-1} \left( \frac{1}{[h_{x}^{-1}(t)]^{N}} + \frac{1}{[h_{y}^{-1}(t)]^{N}} \right) e^{- c_{0\star}( \frac{h_{x}(\mid x-y \mid)}{t})^{\frac{1}{1- \alpha}}} \\
&\leq   \Gamma(x, y , t) \\
&\leq C_{0\star}^{-1} \left( \frac{1}{[h_{x}^{-1}(t)]^{N}} + \frac{1}{[h_{y}^{-1}(t)]^{N}} \right) e^{- C_{0\star} ( \frac{h_{x}(\mid x-y \mid)}{t})^{\frac{1}{1- \alpha}}}
\end{eqnarray*}
  for $x, y \in \mathbb{R}^{N}$, $t > 0$, and $\alpha \in \{a, b\}$. Where $h_{x}^{-1}$ denotes the inverse function of {\small $h_x$ }.

%\item[$(\mbox{K}_4)$] There exist a constant $c=c(N,\alpha) >0$ such that
%\begin{equation*}
 %   0 < \Gamma(x,y,t) \leq c t^{- \frac{N}{2 - \alpha}}
%\end{equation*}
%for $x,y \in \mathbb{R}^N$ and $t>0$.
\end{itemize}
~~\\

Also, by estimates $(2.11)$-$(2.12)$ in \cite{Fujish}, we have
\begin{itemize}
\item[$(\mbox{K}_{4})$] $$\displaystyle \int_{|x| \leq t^{\frac{1}{2-\alpha}} } \Gamma(x,y,t) dx \geq C,$$
for all $|y| \leq t^{\frac{1}{2-\alpha}},$ and some constant $C>0.$
\item[$(\mbox{K}_{5})$] $$\Gamma(x,y,t) \geq C t^{- \frac{N}{2-\alpha}},$$
for $|x|, |y| \leq t^{\frac{1}{2-\alpha}}$, $t>0$, and some constant $C>0.$
\end{itemize}
\begin{rem} Notice that a consequence of property $(\mbox{K}_{3})$ is the nonnegativity of the fundamental solution $\Gamma.$
\end{rem}

We will use the following results to show the global-in-time existence of the solutions to \eqref{eq2.1}.
\begin{prop}[\cite{Fujish}]
.\begin{itemize}
\item[$(G_1)$] Let $\phi \in L^{q_1}(\mathbb{R}^N)$ and $1 \leq q_1 \leq q_2 \leq \infty,$ then
  \begin{equation*}\label{G1}
  \|S(t)\phi \|_{q_2} \leq c_1 t^{- \frac{N}{2 - \alpha} \left(\frac{1}{q_1}- \frac{1}{q_2}\right)} \|\phi\|_{q_1}, \ t>0.
  \end{equation*}
Where the constant $c_1$ can be taken so that it depends only on $N$, $\alpha \in \{a, b\}$.

\item[$(G_2)$] Let $\phi \in L^{q_1, \infty}(\mathbb{R}^N)$ with $1 < q_1 \leq q_2 \leq \infty$, then
\begin{equation*}\label{G2}
\| S(t) \phi \|_{q_2, \infty} \leq c_2 t^{- \frac{N}{2 - \alpha} \left( \frac{1}{q_1} - \frac{1}{q_2} \right) } \| \phi \|_{q_1 , \infty}, t >0.
\end{equation*}
Where the constant $c_2$ can be taken so that it depends only on $q_1, N$, and $\alpha \in \{ a , b \}$. In particular, $c_2$ is bounded in $q_1 \in (1 + \varepsilon , \infty)$ for any fixed $\varepsilon >0$ and $c_2 \to \infty$ as $q_1 \to 1$.
\end{itemize}
\end{prop}
Note that $L^{\infty,\infty}(\mathbb{R}^N)=L^{\infty}(\mathbb{R}^N).$
~~\\

Another tool used is the following interpolation result in Lorentz space.
\begin{prop}[\cite{grafakos}]
Let $1 \leq r_0 \leq r_2 \leq r_1 \leq \infty$ be such that $\frac{1}{r_2} = \frac{\theta}{r_0} + \frac{1 - \theta}{r_1}$, for $\theta \in [0,1]$. Then
\begin{equation}\label{Lorentz1}
\|f\|_{r_2,\infty} \leq \|f\|_{r_0,\infty}^\theta \|f\|_{r_1 , \infty}^{1 - \theta}, \mbox{ for } f \in L^{r_0 , \infty} \cap L^{r_1 , \infty}
\end{equation}
\end{prop}
~~\\

For proof of the non-global existence of solutions to \eqref{eq2.1}, we will use the following results
\begin{lem}[\cite{Fujish}, Lemma 2.4] \label{Fujish2}
Assume that $\omega$ satisfies either $(A)$ or $(B)$. Let $\phi \in L^{\infty}(\mathbb{R}^{N}),$ $\phi \geq 0,$ and $\phi \neq 0$.  Then there exists a positive constant $C(\alpha,N)$, depending only on $\alpha $ and $N$, such that
\begin{equation*}
 S(t)\phi(x)\geq C(\alpha, N)^{-1} t^{-\frac{N}{2- \alpha}} \int_{| y | \leq t^{\frac{1}{2 - \alpha}}} \phi(y)dy,
\end{equation*}
for $ | x | \leq t^{\frac{1}{2- \alpha}}$ and $t>0,$ where $\alpha$ is defined  by $\alpha = a$ in the case (A) and $\alpha = b$ in the case (B).
\end{lem}
~~\\

\begin{lem}\label{13.1} Assume that $\omega$ satisfies either $(A)$ or $(B)$. If $u_0 \in L^{\infty}(\mathbb{R}^N)$ is a nonnegative function and $p\geq1$, then $ [S(t)u_0]^p\leq S(t)u_{0}^{p}$. If $ 0<p<1$, then $ [S(t)u_0]^p\geq S(t)u_0^p$.
\end{lem}
\begin{proof}
Since the fundamental solution $\Gamma \geq 0$ and
$$  \displaystyle \int_{\mathbb{R}^{N}} \Gamma(x , y , t) dy=1,$$
we can use Jensen's inequality and so get
$$  S(t)u_{0}^{p}(x) = \int_{\mathbb{R}^{N}} \Gamma(x,y,t) u_{0}^{p}(y) dy \geq \left(\int_{\mathbb{R}^{N}} \Gamma(x,y,t) u_{0}(y) dy \right)^{p}= [S(t)u_0(x)]^p. $$
Assume now that $0<p<1$. Since $u_0^p \in L^{\infty}(\mathbb{R}^{N})$ is a nonnegative function, the conclusion follows as the anterior case replacing  $p$ by $1/p$.
\end{proof}
~~\\

\section{Nonglobal Existence}

We need the following Proposition to prove the first part of our main result.

\begin{prop}\label{p3.2} Assume that $\omega$ satisfies either $(A)$ or $(B)$, and $u_0, v_0 \in L^\infty(\mathbb{R}^N)$ with $u_0, v_0 \geq 0$. Suppose that $(u,v) \in \left[L^{\infty}((0,T), L^{\infty}(\mathbb{R}^{N})) \right]^{2}$ is a solution of $(\ref{eq2.1})$ with $0 < T \leq \infty,$ and $p,q > 0$ with $pq > 1$.
Then, there exists a constant $ C^{\star} >0$ (which depends only on $p,q,r,$ and $s$),  such that
\begin{equation} \label{Estimate22}
\begin{array}{rlllll}
t^{\frac{(r+1) + (s+1) p}{pq-1}} \|S(t)u_{0}\|_{\infty} & \leq& C^{\star} &\mbox{ if } q > 1
\\ \noalign{\medskip}
t^{ q\frac{(r+1) + (s+1) p}{pq-1} } \|S(t)u_{0}^{q}\|_{\infty} &\leq& C^{\star} & \mbox{ if } 0 < q < 1,
\\ \noalign{\medskip}
t^{\frac{(s+1) + (r+1) q}{pq-1}} \|S(t)v_{0}\|_{\infty} & \leq& C^{\star} &\mbox{ if } p > 1
\\ \noalign{\medskip}
 t^{ p\frac{(s+1) + (r+1) q}{pq-1} } \|S(t)v_{0}^{p}\|_{\infty} & \leq& C^{\star}  &\mbox{ if } 0 < p < 1.
\end{array}
\end{equation}
for all $t\in[0, T).$
\end{prop}

\begin{proof}
We will first prove the first inequality in (\ref{Estimate22}). To do this, we will show the following estimate
\begin{equation}\label{T.uno}
u(t) \geq C_k t^{(\beta^k-1) \gamma_1} [S(t) u_{0}]^{\beta^k}
\end{equation}
for all $t \in (0,T)$ and $k \in \mathbb{N}.$ Where $C_0=1$, $\beta = pq$, $ \gamma_1 = \frac{(r+1) + (s+1) p}{pq-1} $ and
\begin{equation}\label{T.tre}
\begin{array}{ll}
 C_k = C_{k-1}^{\beta}[(\beta^{k-1} -1) q \gamma_1+ s+1]^{-p} [(\beta^{k-1} -1) \gamma_1 \beta + p (s+1) +(r+1)]^{-1},
\end{array}
\end{equation}
for $k \in \mathbb{N}$. Indeed, we argue by induction. From (\ref{Mild}) and property $(K_3)$, we get that $u(t) \geq S(t)u_{0}$, for $t>0 $, thus (\ref{T.uno}) holds for $k=0$. Now, assume that (\ref{T.uno}) holds for $k \geq 1$, then from (\ref{Mild}), $(K_1),$ $(K_2),$ $(K_3),$ and Lemma \ref{13.1}, we have
\begin{equation} \label{BLOWWW}
\begin{array}{ll}
v(t) &\geq  \displaystyle \int_0^t S(t-\sigma) \sigma^{s} [u(\sigma)]^{q}d\sigma
\\ \noalign{\medskip}
& \geq  \displaystyle \int_0^t S(t-\sigma) \sigma^{s} \left[C_k \sigma^{(\beta^k-1) \gamma_1} [S(\sigma) u_{0}]^{\beta^k} \right]^{q}d\sigma
\\ \noalign{\medskip}
&\geq C_k^{q}  \displaystyle \int_0^t  \sigma^{(\beta^k-1) \gamma_1 q +s}S(t-\sigma)[S(\sigma)u_{0}]^{q \beta^k}d\sigma
\\ \noalign{\medskip}
&\geq C_k^{q} [S(t)u_{0}]^{q \beta_k}  \displaystyle \int_0^t \sigma^{(\beta^k-1) \gamma_1 q + s}d\sigma
\\ \noalign{\medskip}
&= C_{k,1} t^{(\beta^k-1) \gamma_1 q+ s+1} [S(t)u_{0}]^{q \beta^k}
\end{array}
\end{equation}

for $t>0,$ where $ C_{k,1}= C_k^{q}/ ((\beta^k-1) \gamma_1 q+s+1).$ Similarly, from (\ref{BLOWWW}), we obtain
$$
\begin{array}{ll}
u(t) &\geq  \displaystyle \int_0^t S(t-\sigma) \sigma^{r} [v(\sigma)]^{p}d\sigma
\\ \noalign{\medskip}
& \geq  \displaystyle \int_0^t S(t-\sigma) \sigma^{s} \left[ C_{k,1} \sigma^{(\beta^k-1) \gamma_1 q+ s+1} [S(\sigma)u_{0}]^{q \beta^k} \right]^{p}d\sigma
\\ \noalign{\medskip}
&\geq  C_{k,1}^{p}[S(t)u_{0}]^{\beta^{k+1}}  \displaystyle \int_0^t \sigma^{(\beta^k-1) \gamma_1 \beta +(s+1) p + r}d\sigma
\\ \noalign{\medskip}
&=C_{k,2} t^{ (\beta^k-1) \gamma_1 \beta +(s+1) p + (r+1) } [S(t)u_{0}]^{\beta^{k+1}}
\end{array}
$$
for $ t>0,$  where $C_{k,2}= C_{k,1}^{p}/ [(\beta^k-1) \gamma_1 \beta +(s+1) p + (r+1)].$ Since that
$$ (\beta^k-1) \gamma_1 \beta +(s+1) p + (r+1) = (\beta^{k+1}-1)\gamma_1, $$
we have
\begin{eqnarray*}
u(t) \geq C_{k,2} t^{ (\beta^{k+1}-1) \gamma_1} [S(t)u_{0}]^{\beta^{k+1}},
\end{eqnarray*}
for $ t > 0$. Also, denoting $C_{k+1}= C_{k,2}$ and inserting the value of $C_{k,1}$, we get (\ref{T.tre}).

Now we to show that there exists $\kappa_0 >0 $ such that $C_k \geq \kappa_{0}^{\beta^k}$ for all $k\geq2$. Let $\theta_k = - \beta^{-k} \ln(C_k)$, it is sufficient to prove that the sequence $ \{\theta_k\}_{k \in \mathbb{N}} $ is limited from above. From \eqref{T.tre}, we have
\begin{align*}
\theta_i - \theta_{i-1} &= \beta^{-i} \ln\left( \frac{C_{i-1}^{\beta}}{C_i} \right) \\
 &= \beta^{-i} \ln \left( [(\beta^{i-1} -1) q \gamma_1+ s+1]^{p} [(\beta^{i-1} -1) \gamma_1 \beta + p (s+1) +(r+1)] \right) \\
& \leq \beta^{-i} \ln \left( \left\{\begin{array}{rlll}
 [\gamma_1 (\beta^i - 1)]^{p+1} &\mbox{ if } p>1,\\
q[\gamma_1 (\beta^i - 1)]^2 &\mbox{ if } 0 < p \leq 1
      \end{array}\right. \right)\\
& \leq C \beta^{-i} (i+1),
\end{align*}
this implies $\theta_k - \theta_{1} = \sum_{i=1}^{k} (\theta_i - \theta_{i-1}) \leq C \sum_{i=1}^{k} \beta^{-i} (i+1) < \infty.$ Thus, from this and \eqref{T.uno} we have
$$ u(t)^{1/\beta^{k}} \geq  \kappa_0 \,  t^{\gamma_1(1-1/\beta^{k})} S(t)u_{0},$$
for all $t \in (0, T).$ Since $\beta>1$, letting $k\rightarrow\infty$, we get the first inequality in (\ref{Estimate22}). \\

We argue similarly to the above case to prove the second inequality of \eqref{Estimate22}. We apply Lemma \ref{13.1} iteratively, starting with $v(t)\geq t^{s+1}S(t)u_0^q $, until getting the following inequality
\begin{equation}\label{SecondInequality}
u(t) \geq D_k t^{(\beta^k-1) \gamma_1 }[S(t)u_0^q]^{p\beta^{k-1}}
\end{equation}
for all $t \in (0,T)$ and $k \in \mathbb{N}.$ Where $\beta = pq$, $ \gamma_1 = \frac{(r+1) + (s+1) p}{pq-1},$ and $D_k \geq \eta_{1}^{\beta_k} \, (\eta_1>0)$. So, from (\ref{SecondInequality}), we get
$$ u(t)^{q/\beta^{k}} \geq  \eta_1  t^{q\gamma_1(1-1/\beta^{k})} S(t)u_{0}^q,$$ for all $t \in (0,T)$ and some positive constant $\eta_1$. Then, letting $k$ tend to infinity, we obtain the second inequality of (\ref{Estimate22}).

By symmetry, the other inequalities follow immediately.
\end{proof}
~~\\

The following Corollary is a direct consequence of the above Proposition
\begin{cor} \label{COR}  Assume that $\omega$ satisfies either $(A)$ or $(B)$, and $u_0, v_0 \in L^\infty(\mathbb{R}^N)$ with $u_0, v_0 \geq 0$. If $(u,v) \in \left[L^{\infty}((0,\infty), L^{\infty}(\mathbb{R}^{N})) \right]^{2}$ is a global-in-time solution of $(\ref{eq2.1})$.
Then, there exists a constant $ C^{\star\star} >0$ (which depends only on $p,q,r,$ and $s$), such that
\begin{equation*}
\begin{array}{rlllll}
t^{\frac{(r+1) + (s+1) p}{pq-1}} \|S(t)u(t)\|_{\infty}& \leq& C^{\star\star} &\mbox{ if } q > 1
\\ \noalign{\medskip}
t^{ q\frac{(r+1) + (s+1) p}{pq-1} } \|S(t)u(t)^{q}\|_{\infty}& \leq& C^{\star\star}  &\mbox{ if } 0 < q < 1,\\ \noalign{\medskip}
t^{\frac{(s+1) + (r+1) q}{pq-1}} \|S(t)v(t)\|_{\infty}& \leq& C^{\star\star} &\mbox{ if } p > 1
\\ \noalign{\medskip}
t^{ p\frac{(s+1) + (r+1) q}{pq-1} } \|S(t)v(t)^{p}\|_{\infty}& \leq &C^{\star\star}  &\mbox{ if } 0 < p < 1.
\end{array}
\end{equation*}
for all $t\in(0,\infty)$.
\end{cor}
\begin{proof}
Since $(u.v) \in [L^\infty((0,\infty), L^\infty(\mathbb{R}^N))]^2$ is a global-in-time solution of (\ref{eq2.1}), then $(u(t+\sigma),v(t+\sigma))$ for $t > 0$ and for $\sigma >0$ is solution for problem (\ref{eq2.1}) with initial condition $(u(\sigma),v(\sigma))$. Thus, the estimate (\ref{Estimate22}) with $(u(\sigma),v(\sigma))$ instead of $(u_0,v_0)$ is hold. Therefore the result follows by taking $\sigma = t$ in this estimate.
\end{proof}
~~\\

\textbf{Proof of Nonglobal Existence (Theorem \ref{Theorem3} - (i)).} Without loss of generality, we can assume that $\gamma_1=\gamma$. Thus, we have two cases :

\textbf{Case I : $\mathbf{q>1}$.} We argue by contradiction. Suppose that there exists $(u,v),$ a non-trivial global-in-time solution of (\ref{eq2.1}) with initial condition $(u_{0}, v_{0})$, Thus, $u_{0}$ or $v_{0}$ is a non-trivial function. Suppose that $u_{0} \neq 0$, thus by Lemma \ref{Fujish2} and arguing as the proof of Proposition \ref{p3.2}, we have
\begin{equation} \label{BLOW1}
\begin{array}{ll}
u(t) &\geq  [S(t)u_{0}] > 0 \,\, \mbox{ and } \,\, v(t) \geq (s+1)^{-1} [S(t)u_{0}]^{q} t^{s+1} > 0,
\end{array}
\end{equation}
for $t>0$.

Let $w(t):=u(t+\tau)$ and  $z(t):=v(t+\tau)$  for $ t \geq 0 $ and some $\tau\geq1$. Note that, (\ref{BLOW1}) implies that $w(0) \neq  0 , z(0) \neq  0$. Since $(w,z)$ satisfies (\ref{Mild}) with initial condition $(w(0),z(0)) = (u(\tau),v(\tau))$, then by Proposition \ref{p3.2} we have
\begin{equation} \label{BLOW2}
t^{\frac{(r+1) + (s+1) p}{pq-1}} \| S(t) w(0)\|_{\infty} \leq C^{\star},
\end{equation}
for all $t\geq0.$

We can find a non-trivial function $ 0 \leq U_1 \in L^{\infty}(\mathbb{R}^N) $ such that $supp \ U_1 \subset B(t_{0}^{ \frac{1}{2-\alpha}})$ (the ball of center $0$ and radius $t^{\frac{1}{2-\alpha}}_{0}$) for some $t_{0}\geq 1,$ and $0 \leq U_1 \leq w(0).$ By Lemma \ref{Fujish2}, we have
\begin{equation}\label{BLOW5}
S(t)U_1(x) \geq C M  t^{- \frac{N}{2 - \alpha}},\ \ \ \ \ M:= \int_{B(t_{0}^{ \frac{1}{2-\alpha}})} U_1(y) dy,
\end{equation}
for $t \geq t_0,$ $ | x | \leq t^{\frac{1}{2 - \alpha}},$ and $C>0$.

Let us first assume that $ \gamma_1  > \frac{N}{2 - \alpha}.$ From (\ref{BLOW5}) and $(K_3)$, we have
$$ t^{\gamma_1} \| S(t) w(0) \|_{\infty} \geq  t^{\gamma_1} \| S(t) U_1  \|_{\infty}  \geq C M t^{\gamma_1 - \frac{N}{2 - \alpha}},$$
for  all $t\geq t_0.$ But this contradicts \eqref{BLOW2}.

Now consider $ \gamma_1 = \frac{N}{2 - \alpha}.$ Computing similarly as in (\ref{T.uno}), we have
\begin{equation} \label{BLOWIII}
z(t) \geq C t^{s+1} [S(t)w(0)]^{q},
\end{equation}
for $t>0$ and some constant $C>0.$ On the other hand, from \eqref{BLOW5}, we have
\begin{equation} \label{blow11}
[S(t)w(0)](x) \geq C t^{- \frac{N}{2 - \alpha}} = C t^{- \gamma_1} ,
\end{equation}
for $t \geq t_0,$ and $ | x | \leq t^{\frac{1}{2 - \alpha}}$.

Note that, $t+1-\sigma \leq t $ and $\sigma \leq t + 1 - \sigma $ for $1\leq \sigma \leq t/2 $. Thus, from (\ref{Mild}), $(K_3)$, $(K_4)$, $(K_5)$, (\ref{BLOWIII}), and (\ref{blow11}), we get
\begin{equation} \label{BLOW4}
\begin{array}{ll}
\displaystyle \int_{ | x | \leq (t+1)^{\frac{1}{2 - \alpha}}} w(x,t+1) dx
\\ \noalign{\medskip}
\geq \displaystyle \int_{ | x | \leq t^{\frac{1}{2 - \alpha}}} w(x,t+1) dx
\\ \noalign{\medskip}
\geq \displaystyle \int_{ | x | \leq t^{\frac{1}{2 - \alpha}}} \int_{1}^{\frac{t}{2}} \int_{ | y | \leq (t+1-\sigma)^{\frac{1}{2 - \alpha}}} \sigma^{r} \ \Gamma(x,y,t+1-\sigma) z(y,\sigma)^{p} dy d \sigma dx
\\ \noalign{\medskip}
\geq \displaystyle   \int_{1}^{\frac{t}{2}} \int_{ | y | \leq (t+1-\sigma)^{\frac{1}{2 - \alpha}}} \sigma^{r} \left(  \int_{ | x | \leq (t+1-\sigma)^{\frac{1}{2 - \alpha}}} \Gamma(x,y,t+1-\sigma) dx \right)  z^{p} dy d \sigma
\\ \noalign{\medskip}
\geq \displaystyle  C  \int_{1}^{\frac{t}{2}} \int_{| y | \leq (t+1-\sigma)^{\frac{1}{2 - \alpha}}} \sigma^{r} (\sigma^{s+1} [S(\sigma)w(0)]^{q})^{p} dy d \sigma \\
\geq \displaystyle  C  \int_{t_0}^{\frac{t}{2}} \int_{ | y | \leq (t+1-\sigma)^{\frac{1}{2 - \alpha}}} \sigma^{r + (s+1)p} [S(\sigma)w(0)]^{\beta - 1} [S(\sigma)w(0)] dy d \sigma
\\ \noalign{\medskip}
\geq \displaystyle   C  \int_{t_0}^{\frac{t}{2}}  \sigma^{r + (s+1)p} \cdot \sigma^{-(\beta-1)\gamma_1} \left(\int_{ | y | \leq \sigma^{\frac{1}{2 - \alpha}}} \sigma^{-\gamma_1} dy\right) d \sigma
\\ \noalign{\medskip}
\geq C \ \displaystyle  \int_{t_0}^{\frac{t}{2}} d \sigma, \mbox{ for all } t>0 \mbox{ sufficiently large } (t>2t_0\geq2).
\end{array}
\end{equation}
By (\ref{BLOW4}), we see that for every $R>0$ it is possible to find $t_{2} > 1$ such that the function $U_2$ defined by $U_2(x) := w(x, t_2) \in L^{\infty}(\mathbb{R}^N)$ satisfies
\begin{equation} \label{BLOW9}
\begin{array}{rll}
\displaystyle \int_{| x | \leq t^{\frac{1}{2 - \alpha}}_{2}} U_{2}(x) dx \geq R.
\end{array}
\end{equation}
Now consider $ (w_1(t),z_1(t)) = (w(t + t_{2}), z(t + t_{2})).$ Note that $(w_1(t),z_1(t))$ is a global-in-time solution of  problem (\ref{Mild}) with initial condition $(w_1(0),z_1(0)) = (U_2(x), z( t_{2}))$. Therefore, from Proposition \ref{p3.2}, we have
\begin{equation} \label{BLOW10}
 t^{\gamma_1} \| S(t) U_2\|_{\infty} \leq C^{\star}, \ \ \ for \ all \ t\geq0.
\end{equation}
On the other hand, from (\ref{BLOW9}) and Lemma \ref{Fujish2}, we have
$$ S(t)U_2(x) \geq C(\alpha,N)^{-1} R t^{-\frac{N}{2- \alpha}},$$
for $ | x |  \leq t^{\frac{1}{2-\alpha}}$ and $t > t_2.$ This implies that
$$ t^{\gamma_1} \| S(t)  U_2 \|_{\infty} = t^{\frac{N}{2-\alpha}} \| S(t)  U_2 \|_{\infty}  \geq C(\alpha,N)^{-1} R,$$
for all $t > t_0.$ This contradicts inequality \eqref{BLOW10} due to arbitrariness of $R>0.$ \\

\textbf{Case II : $\mathbf{q<1}$.} We argue by contradiction. Suppose that there exists a global-in-time solution $(u,v) $ of problem (\ref{eq2.1}) with initial condition $(u_0,v_0) \in [L^\infty(\mathbb{R}^N)]^2$, $u_0,v_0 \geq 0;$ without loss of generality, we can assume that $u_0 \neq 0 $ and $v_0 \neq 0  $ (see (\ref{BLOW1})).

Suppose that $ \gamma_1 > \frac{N}{2 - \alpha} $. We can find a non-trivial function $ 0 \leq U_3 \in L^{\infty}(\mathbb{R}^N) $ such that $supp \ U_3 \subset B(t_{0}^{ \frac{1}{2-\alpha}})$ for some $t_{0}\geq 1$ and $0 \leq U_3 \leq u_0$. Thus, arguing similarly as in (\ref{BLOW5}), and since $ \Gamma \geq 0$ (by $(K_3)$), we have
$$ u(x,t) \geq [S(t)u_0](x) \geq C t^{-\frac{N}{2 - \alpha} } \mathcal{X}_{t^{\frac{1}{2-\alpha}} }(x),$$
for $t\geq t_0 $  and some constant $C>0$, where $\mathcal{X}_{t^{\frac{1}{2-\alpha}} }$ is the characteristic function on the ball of center $0$ and radius $t^{\frac{1}{2-\alpha}}$. It follows from here that
\begin{equation} \label{BLOWQ}
t^{q \gamma_1} \| S(t) u(t)^{q}\|_{\infty} \geq C  t^{q (\gamma_1 - \frac{N}{2 - \alpha}) } S(t) \mathcal{X}_{t^{\frac{1}{2-\alpha}} }(x).
\end{equation}
for $t \geq t_0 $. Besides, by $(K_5)$, we have
\begin{equation} \label{BBBLOW}
S(t)\mathcal{X}_{t^{\frac{1}{2-\alpha}} }(x) \geq  \displaystyle \int_{ | y | < t^{\frac{1}{2 - \alpha}}} \Gamma(x,y,t) dy \geq C t^{- \frac{N}{2-\alpha}}\cdot t^{ \frac{N}{2-\alpha}},
\end{equation}
for all $| x | \leq t^{\frac{1}{2-\alpha}} $ and $t>0$. Thus, (\ref{BLOWQ}) contradicts the second inequality in Corollary \ref{COR}.

Now assume $ \gamma_1 = \frac{N}{2 - \alpha}.$ Proceeding similarly as in estimates (\ref{BLOWIII}) and (\ref{BLOWQ}), we get that
$$ v(t) \geq C t^{s+1} S(t)u^q(t),$$
for all $t>0,$ and
$$ t^{s+1} S(t) u(x,t)^q \geq C t^{s+1} \cdot t^{-q \gamma_1} \ \ (\mbox{see } \eqref{BBBLOW} ),$$
for $ | x | \leq t^{\frac{1}{2-\alpha}} $ and $t> t_0$, respectively. From here and proceeding as in the obtention of \eqref{BLOW4}, we have
\begin{equation*}
\begin{array}{ll}
 \displaystyle \int_{| x | \leq t^{\frac{1}{2 - \alpha}}} u(x,t+1) dx \\
 \noalign{\medskip}
 \geq C \displaystyle \int_{1}^{t/2} \int_{| y | \leq (t+1-\sigma)^{\frac{1}{2 - \alpha}}} \sigma^{r} (\sigma^{s+1} [S(\sigma)u(y,\sigma)^{q}])^{p} dy \ d \sigma \\ \noalign{\medskip}
\geq \displaystyle C  \int_{t_0}^{t/2} \int_{| y | \leq \sigma^{\frac{1}{2 - \alpha}}} \sigma^{r + (s+1)p} [\sigma^{-\gamma_1}]^{\beta - 1} [\sigma^{-\gamma_1}] dy \ d \sigma \\ \noalign{\medskip}
\geq \displaystyle C  \int_{t_0}^{t/2}  \sigma^{r + (s+1)p} \cdot \sigma^{-(\beta-1)\gamma_1} \left(\int_{| y | \leq \sigma^{\frac{1}{2 - \alpha}}} \sigma^{-\gamma_1} dy\right) \ d \sigma
\\ \noalign{\medskip}
\geq C \displaystyle \int_{t_0}^{t/2} d \sigma, \mbox{ for all } t>0 \mbox{ sufficiently large. } (t>2t_0\geq2).
\end{array}
\end{equation*}
Thus, from Corollary \ref{COR}, and arguing similarly as in Case II, we obtain a contradiction.

\section{Global Existence}

\subsection{Local Existence}

We first establish the local existence of solutions when $p>1$ and $q>1.$ Later, we show the local existence for the general case using an approximations method (see the proof of Corollary \ref{Corollary}). \\

\begin{lem}[Comparison Principle] \label{Comparison} Assume either (A) or (B) is in force, and $(u_{0,i},v_{0,i}) \in [L^\infty(\mathbb{R}^N)]^2 \, (i=1,2).$ Let $f,g:[0,\infty) \rightarrow [0,\infty)$ locally Lipschitz continuous functions, $r,s>-1,$ and $(u_i,v_i) \in L^{\infty}((0,T), L^{\infty}(\mathbb{R}^N)) \, (i=1,2)$ such that
\begin{equation*}
\begin{array}{rllllll}
 u_i(x,t) &=& \displaystyle \int_{\mathbb{R}^N} \Gamma(t,x,y) u_{0,i}(y)dy + \displaystyle \int_{0}^{t} \int_{\mathbb{R}^N} \Gamma(t - \sigma,x,y) \sigma^{r} f(v_i(y,\sigma)) dy d \sigma,&
 \\ \noalign{\medskip}
 v_i(x,t) &=& \displaystyle \int_{\mathbb{R}^N} \Gamma(t,x,y) v_{0,i}(y)dy  + \displaystyle \int_{0}^{t}\int_{\mathbb{R}^N} \Gamma(t - \sigma,x,y) \sigma^{s} g(u_i(y,\sigma)) dy d\sigma, &
\end{array}
\end{equation*}
for almost $x \in \mathbb{R}^N$ and $t>0$. If $ u_{0,1} \leq u_{0,2}$ and $v_{0,1}\leq v_{0,2}$, then $u_1(t) \leq u_2(t)$ and $v_1(t)\leq v_2(t)$  for all $t \in (0,T) $.
\end{lem}
\begin{proof} Note that it is sufficient to show that $ [u_i - v_i]^{+} = 0 \, (i=1,2).$ Let $M_0 = \max \{ \|u_i(t)\|_{\infty}, \|v_i(t)\|_{\infty}: t \in [0,T], \, i=1,2 \} .$ Since that $u_{0,1}\leq u_{0,2} $ and $v_{0,1}\leq v_{0,2}$, from $(K_2)$ we have
\begin{eqnarray*}
u_1(t) - u_2(t) & \leq  \displaystyle \int_{0}^{t} S(t- \sigma)  \sigma^r [f(v_1(\sigma)) - f(v_2(\sigma))  ]  d\sigma, \\
v_1(t) - v_2(t) & \leq \displaystyle  \int_{0}^{t} S(t- \sigma) \sigma^s [g(u_1(\sigma)) - g(u_2(\sigma))  ] d\sigma.
\end{eqnarray*}
Thus, since that $ f $ and $g $ are nondecreasing and Lipschitz continuous on $[0,M_0] $, it follows from $(G_1)$ that
\begin{eqnarray*}
\| [u_1(t) - u_2(t)]^{+} \|_{\infty} &  \leq & C \int_{0}^{t} \sigma^r \| [v_1(\sigma) - v_2(\sigma)]^{+} \|_{\infty}  ~ d\sigma, \\
\| [v_1(t) - v_2(t)]^{+} \|_{\infty} &  \leq & C \int_{0}^{t} \sigma^s \| [u_1(\sigma) - u_2(\sigma)]^{+} \|_{\infty}  ~ d\sigma,
\end{eqnarray*}
The Lemma is now a direct consequence of Gronwall's inequality (for example, see \cite{Webb}).
\end{proof}
~~\\

\begin{thm} \label{Fujish3} Suppose that $p,q >1$ and assume that $\omega$ satisfies either $(A)$ or $(B)$, and $(u_{0},v_{0}) \in [L^{\infty}(\mathbb{R}^N)]^{2},$ $u_0,v_0\geq0$. Then there exists $ T>0$ and a constant $C_0 >0$ such that problem (\ref{eq2.1}) possesses a unique solution $(u,v)$ on $(0,T)$ satisfying
\begin{equation}\label{Local}
\sup_{0<t<T} ( \|u(t)\|_{\infty} + \|v(t)\|_\infty ) \leq C_{0} (\|u_{0}\|_{\infty} +  \|v_{0}\|_{\infty} ).
\end{equation}
\end{thm}
\begin{proof}
For any $(u_0, v_0) \in [L^\infty(\mathbb{R}^N)]^2$, $u_0, v_0 \geq 0$. We define the sequences $\{ u_n\}_{n \geq 1}$ and $\{ v_n \}_{n \geq 1}$ by
\begin{equation*}
u_1(x,t) = \int_{\mathbb{R}^N } \Gamma(x,y,t)u_0(y) dy, \ v_1(x,t) = \int_{\mathbb{R}^N } \Gamma(x,y,t)v_0(y) dy
\end{equation*}
\begin{equation}\label{desig24}
\begin{array}{ll}
u_{n+1}(x,t) = u_1(x,t) + \displaystyle \int_0^t \sigma^r \int_{\mathbb{R}^N} \Gamma(x,y,t-\sigma)v_n(y,s)^p dy d\sigma, \ n= 1,2, \cdots,
\\ \noalign{\medskip}
v_{n+1}(x,t) = v_1(x,t) + \displaystyle \int_0^t  \sigma^s \int_{\mathbb{R}^N} \Gamma(x,y,t-\sigma)u_n(y,s)^q dy d\sigma, \ n= 1,2, \cdots ,
\end{array}
\end{equation}
for almost all $x \in \mathbb{R}^N$ and all $t>0$. The sequences are non-negative and non-decreasing, that is,
\begin{equation}\label{d25}
0\leq u_n(x,t) \leq u_{n+1}(x,t) \mbox{ and } 0 \leq v_n(x,t) \leq v_{n+1}(x,t)
\end{equation}
for almost all $x \in \mathbb{R}^N$, $t>0,$ and all $n \in \mathbb{N}$. This is clear since $\Gamma$, $u_0,$ and $v_0$ are non-negative functions ($\Gamma$ is nonnegative by $(K_3)$ property). Thus, we write the limit of the functions.
\begin{equation}\label{d26}
u_\infty(x,t) = \lim\limits_{n \to \infty} u_n(x,t), \, \,\,\, v_\infty = \lim\limits_{n \to \infty} v_n(x,t).
\end{equation}
Furthermore, we see that $u_\infty(x,t), v_\infty(x,t) \in [0,\infty]$.

Now, we show that the sequences $\{ u_n \}_{n \geq 1}$ and $\{ v_n \}_{n \geq 1}$ are bounded in a small interval of time, that is,
\begin{equation}\label{desig31}
\sup\limits_{0<t<T} (\|u_n(t)\|_{\infty} + \|v_n(t)\|_{\infty}) \leq 2 c_1 (\|u_0\|_{\infty} + \|v_0\|_{\infty})
\end{equation}
for all $n \in \mathbb{N}$ and some $T>0$ small enough. Let $T>0,$ and we argue by induction. For $n=1$, the inequality \eqref{desig31} is hold, and this is due to $(G_1)$. Let us suppose \eqref{desig31} holds, for some $k \in \mathbb{N}$, that is,
\begin{equation*}
\sup\limits_{0<t<T} (\|u_k(t)\|_{\infty} + \|v_k(t)\|_{\infty}) \leq 2 c_1 (\|u_0\|_{\infty} + \|v_0\|_{\infty}).
\end{equation*}
Then, by $(G1)$, we have
\begin{equation}\label{desig32}
\begin{array}{ll}
\|u_{k + 1}(t) \|_{\infty }
& \leq \|u_1(t)\|_{\infty} + \displaystyle \int_0^t \sigma^r \|S(t-\sigma) v_{k}(\sigma)^p \|_{\infty} d\sigma
\\ \noalign{\medskip}
& \leq c_1 \|u_0\|_{\infty} + c_1 \displaystyle \int_0^t \sigma^r  \|v_{k}(\sigma) \|_{\infty}^p d\sigma
\\ \noalign{\medskip}
& \leq c_1 \|u_0\|_{\infty} + c_1 (2 c_1(\|u_0\|_{\infty} + \|v_0\|_{\infty}))^p \displaystyle \int_0^t \sigma^r d\sigma,
\end{array}
\end{equation}
for all $t \in (0,T)$. Similarly, we have
\begin{equation}\label{desig33}
\begin{array}{ll}
\|v_{k + 1}(t) \|_{\infty}
& \leq \|v_1(t)\|_{\infty} + \displaystyle \int_0^t \sigma^s \|S(t-\sigma) u_{k}(\sigma)^p \|_{\infty} d\sigma
\\ \noalign{\medskip}
& \leq c_1 \|v_0\|_{\infty} + c_1 \displaystyle \int_0^t \sigma^s  \|u_{k}(\sigma) \|_{\infty}^q d\sigma
\\ \noalign{\medskip}
& \leq c_1 \|v_0\|_{\infty} + c_1 (2 c_1(\|u_0\|_{\infty} + \|v_0\|_{\infty}))^q \displaystyle \int_0^t \sigma^s d\sigma,
\end{array}
\end{equation}
for all $t \in (0,T)$. Thus, the inequality  \eqref{desig31} follows by adding \eqref{desig32} and \eqref{desig33} and then choosing $T$ as small enough.

Finally, by \eqref{d25}, \eqref{d26}, and \eqref{desig31}, we have that the limits functions $u_\infty$ and $v_\infty$ satisfies \eqref{Mild} and
\begin{equation*}
\sup\limits_{0<t<T} (\|u_\infty (t)\|_{\infty} + \|v_\infty (t)\|_{\infty}) \leq 2 c_1 (\|u_0\|_{\infty} + \|v_0\|_{\infty}).
\end{equation*}
Also, by the comparison principle (see Lemma \ref{Comparison}), $(u_\infty,v_\infty)$ is the unique solution to the problem \eqref{eq2.1}. Therefore, Theorem \ref{Fujish3} is hold.
\end{proof}
~~\\

Now we proof the local existence of solutions of problem \eqref{eq2.1}, in the general case, that is, when $p,q >0$ and $p\cdot q>1$.

\begin{cor}[Local Existence] \label{Corollary} Suppose that $p,q >0$ with $p\cdot q>1,$ and assume that $\omega$ satisfies either $(A)$ or $(B)$, and $(u_{0},v_{0}) \in [L^{\infty}(\mathbb{R}^N)]^{2},$ $u_0,v_0\geq0$. Then there exists $ T>0$ and a constant $C_0 >0$ such that problem (\ref{eq2.1}) possesses a solution $(u,v)$ on $[0,T]$ satisfying
\begin{equation*}
\sup_{0<t<T} ( \|u(t)\|_{\infty} + \|v(t)\|_\infty ) \leq C_{0} (\|u_{0}\|_{\infty} +  \|v_{0}\|_{\infty} ).
\end{equation*}
\end{cor}
\begin{proof} We use a known approximation method; for example, see \cite{Aguirre}.  Without loss of the generality, we can assume that $1\leq q$ and $0< p< 1.$ For each $n \in \mathbb{N},$ consider a nondecreasing global Lipschitz function $f_n$ such that
$$
f_n(s)=\left \{
\begin{array}{rl}
0 & \mbox{ if } s = 0,\\
s^p& \mbox{ if } s > \frac{1}{2n},\\
\end{array} \right.
$$
and $|f_n(s_1)- f_n(s_2)| \leq c_n |s_1 - s_2|$  for all  $s_1,s_2 \in [0,\infty).$ Let now see the following approximate problem of \eqref{eq2.1}.
\begin{equation}\label{Aprox}
 \left\{\begin{array}{rlll}
 u_{t}-  \mbox{div} ( \omega(x)\nabla u )  & = & t^r f_n(v)& \mbox{ in } \mathbb{R}^{N} \times (0,T),\\
  v_{t}- \mbox{div} ( \omega(x)\nabla v )   &= & t^s u^{q}& \mbox{ in } \mathbb{R}^{N} \times (0,T),\\
u(0)=  u_{0}   &    & v(0)=  v_{0}+\frac{1}{n} &  \mbox{ in } \mathbb{R}^{N}.
      \end{array}\right.
\end{equation}
Since that $f_n$ is globally Lipschitz with $f_n(0)=0,$ we can argue similarly as in the proof of Theorem \ref{Fujish3} for to obtain a unique nonnegative bounded solution $(u_n(t), v_n(t))$ of \eqref{Aprox}, which besides satisfies \eqref{Local}. Note that, from $(K_1)$ and $(K_2)$ properties, we obtain
\begin{equation}\label{Loc1}
v_k(t) \geq S(t) (v_0 + 1/k) = \int_{\mathbb{R}^N} \Gamma(x,y,t) v_0(y) dy + 1/k \geq  1/k, \mbox{ for all } k.
\end{equation}
Also, by construction, we have that if $n>m $ then $f_n(t) =f_m(t) $ for $t>1/2m.$ Then, from this and \eqref{Loc1}, we have
\begin{eqnarray*}
u_{m}(t) &=& S(t)u_0  + \displaystyle \int_0^t \sigma^r \int_{\mathbb{R}^N} \Gamma(x,y,t-\sigma) f_{m}(v_m(y,\sigma)) dy d\sigma, \\
         &=& S(t)u_0  + \displaystyle \int_0^t \sigma^r \int_{\mathbb{R}^N} \Gamma(x,y,t-\sigma) f_{n}(v_m(y,\sigma)) dy d\sigma, \\
v_{m}(t) &=& S(t)(v_0 + 1/m) + \displaystyle \int_0^t \sigma^s \int_{\mathbb{R}^N} \Gamma(x,y,t-\sigma) (u_m(y,\sigma))^q dy d\sigma.
\end{eqnarray*}
Thus, we have that $(u_m.v_m)$ is also solution of \eqref{Aprox} with initial condition $(u_0, v_0 + 1/m).$ Therefore, from Lemma \ref{Comparison}, we have $u_m \geq u_n$ and $v_m\geq v_n$ for $n>m$. That is, the sequences $\{u_n \}_{n\in \mathbb{N}}$ and $\{v_n \}_{n\in \mathbb{N}}$ are decreasing and bounded below. Thus the result follows by letting $n$ go to $\infty.$

\end{proof}
~~\\

\subsection{Global existence: Proof of Theorem \ref{Theorem3}-(ii) } ~~\\

Without loss of the generality, we can suppose that $0<p<1$ and $p\cdot q >1.$ Consider $\max \{ \|u_{0}\|_{r_{1 \star}, \infty},  \|v_{0}\|_{r_{2 \star}, \infty} \} < \delta,$ where $\delta>0$ will be chosen later small enough.

From \eqref{V1}-\eqref{V4}, we obtain the following estimates:
\begin{equation}\label{2Induc1}
p\gamma_2 = \gamma_1 + (r+1), \ q \gamma_1 = \gamma_2 + (s+1), \ p r_{1\star} > r_{2\star}, \ qr_{2\star} > r_{1\star}.
\end{equation}
Also since that $\gamma < \frac{N}{2 - \alpha},$ we have $r_{1\star},r_{2\star} >1.$

Now, similarly to the proof of Theorem \ref{Fujish3}, we define the following sequence
$$ \{(u^n, v^n)\}_{n \geq 0} $$
defined by $u^0(t) = S(t)u_{0}$, $v^0(t) = S(t)v_{0}$ and
\begin{equation} \label{UGLOBAL}
\begin{array}{ll}
u^{n} (t) = S(t)u_{0}+ \displaystyle \int_0^tS(t-\sigma) \displaystyle \sigma^{r}[v^{n-1}(\sigma)]^{p} d\sigma,
\\ \noalign{\medskip}
v^{n} (t) = S(t)v_{0}+ \displaystyle \int_0^tS(t-\sigma)  \sigma^{s}[u^{n-1}(\sigma)]^{q} d\sigma,
\end{array}
\end{equation}
for all $t>0$. Note that the sequences $\{ u^{n}\}_{n\geq0} $ and $\{ v^{n} \}_{n\geq0} $ are non-decreasing.

By induction, we prove that.
\begin{equation}\label{EEE1}
\begin{array}{rlllll}
&\| u^{n} (t)\|_{r_{1\star}, \infty } & \leq & 2 c_{\star \star} \delta, \\
\noalign{\medskip}
&\| u^{n} (t)\|_{\infty} & \leq&  2 c_{\star \star} \delta \ t^{-\frac{N}{(2 - \alpha) r_{1\star}} }, \\
\noalign{\medskip}
&\| v^{n} (t)\|_{r_{2\star}, \infty } & \leq& 2 c_{\star \star} \delta, \\
\noalign{\medskip}
&\| v^{n} (t)\|_{\infty} & \leq& 2 c_{\star \star} \delta \ t^{-\frac{N}{(2 - \alpha) r_{2\star}} }.
\end{array}
\end{equation}

From $(G_2),$ we have
\begin{equation}\label{EEE2}
\begin{array}{rlllll}
&\|u^0(t)\|_{r_{1\star}, \infty} &\leq& c_{**}  \| u_0 \|_{r_{1 \star}, \infty}, \\
\noalign{\medskip}
&\|u^0(t)\|_{\mu, \infty} &\leq& c_{**} t^{- \frac{N}{2 - \alpha} \left( \frac{1}{r_{1\star}} - \frac{1}{\mu} \right) } \| u_0 \|_{r_{1\star}, \infty}, \\
\noalign{\medskip}
&\|v^0(t)\|_{r_{2\star}, \infty} &\leq& c_{**}  \| v_0 \|_{r_{2 \star}, \infty}, \\
\noalign{\medskip}
&\|v^0(t)\|_{\mu, \infty} &\leq& c_{**} t^{- \frac{N}{2 - \alpha} \left( \frac{1}{r_{2\star}} - \frac{1}{\mu} \right) } \| v_0 \|_{r_{2\star}, \infty},
\end{array}
\end{equation}
for all $t>0$, $\mu \in [r_{1\star},\infty],$ and some constant $c_{**}>0$. This implies that (\ref{EEE1}) is held for $n = 0$.

Now we assume that (\ref{EEE1}) holds for some $n \in \mathbb{N}$. By symmetry we only prove that (\ref{EEE1}) holds for $u^{n+1}$. From \eqref{Lorentz1} and (\ref{EEE1}), we have
\begin{equation}\label{G.estimate1}
\begin{array}{ll}
\| v^n(t) \|_{\mu , \infty} &\leq \| v^n(t)\|_{r_{2\star}, \infty}^{\frac{r_{2\star}}{\mu}} \|v^n(t)\|_\infty^{1 - \frac{r_{2\star}}{\mu}}
\\
\noalign{\medskip}
& \leq 2 c_{**} \delta t^{-\frac{N}{2 - \alpha} \left( \frac{1}{r_{2\star}} - \frac{1}{\mu} \right)}
\end{array}
\end{equation}
for all $t>0$ and $r_{2\star} \leq \mu < \infty$. Also, from \eqref{2Induc1} and \eqref{G.estimate1}, we have
\begin{equation}\label{G.estimate3}
\begin{array}{ll}
\|v^n(t)^p \|_{\eta , \infty} &= \| v^n(t)\|_{\eta p, \infty}^p
\\
\noalign{\medskip}
& \leq (2 c_{**} \delta t^{- \frac{N}{2-\alpha} \left( \frac{1}{r_{2\star}} - \frac{1}{\eta p} \right)} )^p
\\
\noalign{\medskip}
& = C \delta^p t^{\frac{N}{(2-\alpha)\eta} - \frac{N}{(2 - \alpha) r_{1\star}} - (r+1)}
\end{array}
\end{equation}
for any $\eta >1$ with $r_{2\star} \leq \eta p.$ Similarly, from \eqref{2Induc1} and (\ref{EEE1}), we obtain
\begin{equation}\label{G.estimate2}
\begin{array}{ll}
\|v^{n}(t)^{p}\|_{\infty} &= \|v^{n}(t)\|_{\infty}^{p}
\\ \noalign{\medskip}
&\leq  \left( 2 c_{\star \star} \delta \ t^{-\frac{N}{(2 - \alpha) r_{2\star}} } \right)^{p}
\\ \noalign{\medskip}
&=  (2 c_{\star \star} \delta)^{p} t^{- p \gamma_{2}}
\\ \noalign{\medskip}
&= (2 c_{\star \star} \delta)^{p} t^{- \frac{N}{(2 - \alpha) r_{1\star}} - (r + 1) }
\end{array}
\end{equation}
for all $t > 0$.

Thus, by $(G_1)$, $(G_2)$, \eqref{2Induc1}, (\ref{G.estimate3})($\mbox{with } \eta=r_{1\star}$), and \eqref{G.estimate2},  we have
\begin{equation}\label{G.estimate4}
\begin{array}{rll}
\left \| \displaystyle \int_{t/2}^{t} S(t - \sigma) \sigma^{r} v^{n}(\sigma)^{p} d \sigma \right \|_{\infty} &\leq  \displaystyle \int_{t/2}^{t} \sigma^{r}  \| S(t - \sigma) v^{n}(\sigma)^{p} \|_{\infty} d \sigma
\\ \noalign{\medskip}
& \leq C \displaystyle \int_{t/2}^{t} \sigma^{r}  \| v^{n}(\sigma)^{p} \|_{\infty} d \sigma \\
\noalign{\medskip}
& \leq  C \delta^{p} t^{- \frac{N}{(2 - \alpha) r_{1\star}}}
\end{array}
\end{equation}
and
\begin{equation}\label{G,estimate5}
\begin{array}{ll}
\left\| \displaystyle \int_{\frac{t}{2}}^t S(t - \sigma) \sigma^p v^n(\sigma)^p d\sigma \right\|_{r_{1\star}, \infty} & \leq \displaystyle \int_{\frac{t}{2}}^t \sigma^r \| S(t - \sigma) v^n(\sigma)^p \|_{r_{1\star}, \infty} d\sigma
\\ \noalign{\medskip}
& \leq C \displaystyle \int_{\frac{t}{2}}^t \sigma^r \| v^n(\sigma)\|^p_{p r_{1\star}, \infty} d \sigma \\
\noalign{\medskip}
& \leq C \delta^p
\end{array}
\end{equation}
for all $t > 0.$

On the other hand, since that $ t - \sigma \geq t/2 \mbox{ for all } \sigma \in [0, t/2],$ by $(G_2)$ and \eqref{G.estimate3} (with $\eta=\eta_1$, which will be chosen later), we obtain
\begin{equation}\label{G.estimate6}
\begin{array}{rll}
&\left \|  \displaystyle \int_{0}^{\frac{t}{2}} S(t - \sigma) \sigma^{r} v^{n}(\sigma)^{p} d \sigma \right \|_{\infty}  \\
&\leq  \displaystyle \int_{0}^{\frac{t}{2}}  \| S(t - \sigma) \sigma^{r} v^{n}(\sigma)^{p} \|_{\infty} d \sigma
\\ \noalign{\medskip}
&\leq  C  \displaystyle \int_{0}^{\frac{t}{2}} (t - \sigma)^{-\frac{N}{(2 - \alpha) \eta_1}} \sigma^{r}  \| v^{n}(\sigma)^{p} \|_{\eta_1,\infty} d \sigma
\\ \noalign{\medskip}
&\leq  C \delta^{p} t^{-\frac{N}{(2- \alpha) \eta_1}} \displaystyle \int_{0}^{\frac{t}{2}} \sigma^{\frac{N}{(2 - \alpha) \eta_1}  - \frac{N}{(2 - \alpha) r_{1\star}} - 1 } d \sigma
\\ \noalign{\medskip}
&\leq  C \delta^{p} t^{-\frac{N}{(2- \alpha) r_{1 \star}}}
\end{array}
\end{equation}
for some $1 < \eta_1 < r_{1 \star}$ close enough to $r_{1 \star}$ so that $ r_{ 2\star} < \eta_1 p$ (it is possible since that $p \ r_{1 \star}> r_{ 2\star} >1 $).

Analogously (using the above $\eta_1$ again), we obtain
\begin{equation}\label{G.estimate7}
\begin{array}{ll}
& \left\| \displaystyle \int_0^{\frac{t}{2}} S(t - \sigma) \sigma^r v^n(\sigma)^p d \sigma \right\|_{r_{1\star}, \infty} \\
& \leq \displaystyle\int_0^{\frac{t}{2}} \| S(t - \sigma) \sigma^r v^n(\sigma)^p \|_{r_{1\star},\infty} d\sigma
\\
\noalign{\medskip}
& \leq C \displaystyle \int_0^{\frac{t}{2}} (t-\sigma)^{-\frac{N}{2 - \alpha} \left( \frac{1}{\eta_1} - \frac{1}{r_{1\star}} \right)} \sigma^r \|v^n(\sigma)^p \|_{\eta_1,\infty} d\sigma
\\
\noalign{\medskip}
& \leq C t^{-\frac{N}{2 - \alpha} \left( \frac{1}{\eta_1} - \frac{1}{r_{1\star}} \right)}  \displaystyle \int_0^{\frac{t}{2}} \sigma^r \|v^n(\sigma)^p \|_{\eta_1,\infty} d\sigma
\\
\noalign{\medskip}
& \leq C \delta^p
\end{array}
\end{equation}
for all $t>0$.

Then, from \eqref{UGLOBAL}, \eqref{EEE2}, \eqref{G.estimate4}, \eqref{G,estimate5}, \eqref{G.estimate6}, \eqref{G.estimate7}, and taking a $\delta > 0$ sufficiently small, we obtain
\begin{eqnarray*}
t^{\frac{N}{(2 - \alpha) r_{1 \star}}} \|u^{n+1}(t)\|_{\infty} &\leq& c_{\star \star} \delta + C \delta^{p} \leq 2  c_{\star \star} \delta
\\
\|u^{n+1}(t)\|_{r_{1 \star} , \infty} &\leq& c_{\star \star} \delta + C \delta^p \leq 2 c_{\star \star} \delta
\end{eqnarray*}
for all $t > 0,$ where $C>0$ is a constant independent of $n$, $\delta,$ and $t$. Thus, we obtain that (\ref{EEE1}) holds for all $ u^n$ with $n \in \mathbb{N}.$ Arguing similarly, we have that also (\ref{EEE1}) holds for everyone $ v^n$ with $n \in \mathbb{N}.$

Now, from (\ref{EEE1}) and arguing similarly as in the proof of Theorem \eqref{Fujish3}, we obtain that there exists a global-in-time solution (this solution is unique in the particular case when $p>1$ and $q>1$)
$$(u(t),v(t)) = \left( \lim_{n\rightarrow \infty} u^{n}(t),  \lim_{n\rightarrow \infty} v^{n}(t)  \right )$$ of (\ref{eq2.1}) such that
\begin{eqnarray*}
\|u(t)\|_{\infty} \leq 2  c_{\star \star} \delta t^{-\frac{N}{(2 - \alpha) r_{1 \star}}},  & & \|u(t)\|_{r_{1 \star},\infty} \leq  2  c_{\star \star} \delta
\\
\| v(t) \|_{\infty} \leq 2 c_{\star \star} \delta t^{- \frac{N}{(2-\alpha) r_{2 \star}}}, & & \|v(t)\|_{r_{2 \star}, \infty } \leq 2 c_{\star \star} \delta.
\end{eqnarray*}
Using this together with the upper bounded estimate in Corollary \ref{Corollary}, we obtain a constant $C_0 >0$ such that
$$\|u(t)\|_{\infty} \leq C_{0} (t+ 1)^{-\frac{N}{(2 - \alpha) r_{1 \star}}} \, \, \mbox{ and } \, \, \|v(t)\|_{\infty} \leq C_{0} (t+ 1)^{-\frac{N}{(2 - \alpha) r_{2 \star}}},$$ for $t>0$. Also, from this and by \eqref{Lorentz1}, we have
\begin{eqnarray*}
\| u(t) \|_{\mu , \infty} &\leq \| u(t) \|_{r_{1 \star} , \infty}^{\frac{r_{1 \star}}{\mu}} \| u(t)\|_\infty^{1 - \frac{r_{1 \star}}{\mu}} \\
 &\leq C_1 (t+1)^{- \frac{N}{2 - \alpha} \left( \frac{1}{r_{1 \star}} - \frac{1}{\mu} \right)}, \\
\| v(t) \|_{\mu , \infty} & \leq \| v(t) \|_{r_{2 \star} , \infty}^{\frac{r_{2 \star}}{\mu}} \| v(t)\|_\infty^{1 - \frac{r_{2 \star}}{\mu}} \\
&\leq C_1 (t+1)^{- \frac{N}{2 - \alpha} \left( \frac{1}{r_{2 \star}} - \frac{1}{\mu} \right)},
\end{eqnarray*}
for all $\mu$ such that $\max \{ r_{1 \star} , r_{2 \star} \} < \mu \leq \infty$, $t>0$ and some constant $C_1 >0$, thus the proof is concluded.

%\subsection*{Acknowledgments}  

\end{document}